\documentclass[12pt,a4paper,reqno]{amsart}
\usepackage{relsize,exscale}
\usepackage{hyperref}
\newcommand*{\affaddr}[1]{#1}
\newcommand*{\affmark}[1][*]{\textsuperscript{#1}}
\usepackage{graphics,layout,indentfirst,amsmath,amsthm,amssymb}
\newtheorem{theorem}{Theorem}%[section]

\theoremstyle{theorem}

\newtheorem{lemma}[theorem]{Lemma}

\numberwithin{equation}{section}
\makeatletter
\@namedef{subjclassname@2010}{2010 Mathematics Subject
Classification} \makeatother
\begin{document}
\title[On the Convergence of Random Fourier--Jacobi Series in $L_{[-1,1]}^{p,(\eta,\tau)}$ space]
{ On the Convergence of Random Fourier--Jacobi Series in $L_{[-1,1]}^{p,(\eta,\tau)}$ space}
\author[P. Maharana, S. Sahoo]
{P\lowercase{artiswari} M\lowercase{aharana}\affmark[1] \lowercase{and}
S\lowercase{abita} S\lowercase{ahoo}\affmark[2]\\
\affaddr{\affmark[1] D\lowercase{epartment} \lowercase{of} M\lowercase{athematics}, S\lowercase{ambalpur} U\lowercase{niversity}, O\lowercase{disha}, I\lowercase{ndia}}\\
\affaddr{\affmark[2] D\lowercase{epartment} \lowercase{of} M\lowercase{athematics}, S\lowercase{ambalpur} U\lowercase{niversity}, O\lowercase{disha}, I\lowercase{ndia}.
}\\
{\affmark[1] \lowercase{partiswarimath1@suniv.ac.in} }\\
{\affmark[2] \lowercase{sabitamath@suniv.ac.in} } \\
 }
%\date{}
\begin{abstract}
Liu and Liu introduced the random Fourier transform, which is a random Fourier series in Hermite functions, and applied it to image encryption and decryption.
They expected its applications in optics and information technology. These motivated us to look into random Fourier series in orthogonal polynomials.
Recently, we have established the convergence of random Fourier--Jacobi series
$\sum\limits_{n=0}^\infty d_n r_n(\omega)\varphi_n(y),$
where $\varphi_n(y)$ are the orthonormal
Jacobi polynomials $p_n^{(\gamma,\delta)}(y),$ $r_n(\omega)$ are random variables
associated with stochastic processes like the Wiener process, the symmetric stable process, and the scalars $d_n$ are the Fourier--Jacobi coefficients of functions in some classes of continuous functions.
It is observed that the mode of convergence of the random series depends on the choice of the scalars $d_n$
and the stochastic processes.
In this article, we have investigated the scalars $d_n,$ which are chosen to be the Fourier--Jacobi coefficients of functions
in some weighted $L_{[-1,1]}^{p,(\eta,\tau )}$ spaces, so that the random series converges.
Further, the continuity property of the sum functions is studied.


\end{abstract}
\maketitle{}
\section{\textbf{Introduction}}
\setcounter{equation}{0}
The Fourier series was invented in the early 1800s to solve the problem of heat diffusion in a continuous medium.
It was extended to Fourier series in orthogonal polynomials which has a role in mathematical physics \cite{CM}, \cite{MM}, etc.
Later on, it was extended to random Fourier series, which is an inherent part of signal processing and image processing.
In 2007, Liu and Liu \cite{LL,LL1} attempted to define a random Fourier transform in orthogonal Hermite functions,
which is a Fourier--Hermite series with random coefficients.
The random Fourier transform they introduced is
\begin{equation*}
\mathcal{R}[f(y)]:=\sum_{n=0}^\infty d_n \mathcal{R}(\lambda_n)\varphi_n(y),
\end{equation*}
where $\varphi_n(y)$ are orthogonal Hermite functions,
$\mathcal{R}(\lambda_n):=\exp[i\pi\; \mathrm{Random(n)}]$ are randomly chosen values from the unit circle in $\mathbb{C},$
and $d_n$ are the Fourier-Hermite coefficients of some function $f$ in $L^2(\mathbb{R})$ i.e.
\begin{equation*}
d_n:=\int_{-\infty}^\infty f(y)\varphi_n(y) dy.
\end{equation*}
They applied it to image encryption and decryption. They also expected its application in optics and information technology.
It motivated us to look into random Fourier series in orthogonal polynomials.
Recently \cite{PS}, we studied the convergence of random Fourier series
\begin{equation}\label{0.2}
\sum_{n=0}^\infty d_n r_n(\omega)\varphi_n(y)
\end{equation}
in orthogonal Jacobi polynomials $\varphi_n(y).$
The scalars $d_n$ are the Fourier--Jacobi coefficients of a function $f$ in some class of continuous functions and the random variables $r_n(\omega)$
are the Fourier--Jacobi coefficients of a stochastic process.
The convergence of the random Fourier series (\ref{0.2}) is explored in this article by choosing the scalars $d_n$ which are Fourier--Jacobi coefficients of functions in some weighted $L^p$ spaces.

It is known that,
if $X(t,\omega),\;t\in \mathbb{R}$ is a continuous stochastic process with independent increments and $f$ is a continuous function in $[a,b],$ then the stochastic integral
\begin{equation}\label{1.1}
\int\limits_a^bf(t)dX(t,\omega)
\end{equation}
is defined in the sense of probability, and is a random variable (Lukacs \cite[p.~148]{22}).
Further, if $X(t,\omega),\; t \in \mathbb{R}$ is a symmetric stable process of index $\alpha \in[1,2]$ and $f \in L_{[a,b]}^p,$ $p\geq 1,$ then the stochastic integral (\ref{1.1}) is defined in the sense of probability, for $p\geq \alpha$ (c.f. \cite{NPM}).

Consider the space $L_{[-1,1]}^{p,(\eta,\tau)}$ of all measurable functions $f$ on the segment $[-1,1]$  with
the weight function $\rho^{(\eta,\tau)}(y):=(1-y)^\eta(1+y)^\tau ,\;\;\eta,\tau>-1,$ such
that
\begin{equation*}
\int_{-1}^1 |f(y)\rho^{(\eta,\tau)}(y)|^p dy < \infty.
\end{equation*}
This space is equipped with norm
\begin{equation*}
||f||_{L_{[-1,1]}^{p,(\eta,\tau)}}=\Big\{\int_{-1}^1 |f(y) \rho^{(\eta,\tau)}(y)|^p dy\Big\}^{\frac{1}{p}},\; for \; p \geq 1.
\end{equation*}
If $f\in L_{[-1,1]}^{p,(\eta,\tau)},\;p \geq 1,$ i.e. $f\rho^{(\eta,\tau)} \in L_{[-1,1]}^p,\;\eta,\tau >-1,$ then the stochastic integral
\begin{equation*}
\int\limits_{-1}^1 f(t)\rho^{(\eta,\tau)}(t)dX(t,\omega)
\end{equation*}
will exist in probability, for $p \geq \alpha \geq 1 .$

In particular, if $f(t)$ is the orthonormal Jacobi polynomial $p_n^{(\gamma,\delta)}(t)$ with respect to Jacobi weight $\rho^{(\gamma,\delta)}(t):={(1-t)}^\gamma{(1+t)}^\delta$ on $[-1,1],$
then the stochastic integrals
\begin{equation}\label{2.2}
A_n(\omega):=\int\limits_{-1}^1
p_n^{(\gamma,\delta )}(t)\rho^{(\eta,\tau)}(t)dX(t,\omega)
\end{equation}
exist, for $\gamma,\delta>-1,\eta,\tau\geq 0,$ and are random variables. These $A_n(\omega)$ are called the Fourier--Jacobi coefficients of the symmetric stable process $X(t,\omega)$ and are not independent (see Theorem \ref{L1.1}). But, they are independent, if these are associated with the Wiener process (see Theorem \ref{L2.2}).
In our work, the random coefficients $r_n(\omega)$ in the series (\ref{0.2})
are considered to be these $A_n(\omega),$ which are dependent or independent depending on the choice of the stochastic process $X(t,\omega)$ as the symmetric stable process or the Wiener process respectively.
The scalars $d_n$ are chosen to be the Fourier--Jacobi coefficients $a_n$ of a function $f$ in $L_{[-1,1]}^{p,(\eta,\tau)}$ defined as
\begin{equation}\label{1.3}
a_n:=\int_{-1}^1 f(t) p_n^{(\gamma,\delta)}(t)\rho^{(\gamma,\delta)}(t)dt,\; \gamma,\delta>-1.
\end{equation}
A key insight of this paper is to
 study the convergence of the random series (\ref{0.2})
in Jacobi polynomials associated with stochastic processes like the symmetric stable process and the Wiener process.
We also find the conditions on $\eta,\tau,\gamma,\delta$ as well as investigate
 the weighted spaces $L_{[-1,1]}^{p,(\eta,\tau)},$ so that the
  Fourier--Jacobi coefficients $a_n$ of its functions makes the random series (\ref{0.2}) to
converge. Further, the continuity property of the sum functions is also studied.

This article is structured as follows:\\
Section 2 establishes the convergence of random Fourier--Jacobi series (\ref{0.2}) associated with symmetric stable process $X(t,\omega)$ of index $\alpha \in (1,2],$ as well as for $\alpha=1.$ Section 3 is devoted to study the random Fourier--Jacobi series (\ref{0.2}) associated with the Wiener process.
In this case, the orthogonal polynomials $\varphi_n(y)$ are considered to be the modified Jacobi polynomials $q_n^{(\gamma,\delta)}(y)$ (see equation (\ref{0.4})) in the segment $[0,1].$
It is shown that the random Fourier--Jacobi series (\ref{0.2}) converges in quadratic mean. Moreover,
under a strong condition on the scalars $d_n,$
the convergence
is upgraded to almost sure convergence.
The sum functions of the random Fourier--Jacobi series (\ref{0.2}) associated with stochastic processes are found to be stochastic integrals.
The continuity property of the sum functions is proved in Section 4. It is obtained that the sum functions of the random series (\ref{0.2}) associated with the symmetric stable process and the Wiener process are weakly continuous in probability and continuous in quadratic mean, respectively. Further, the almost sure continuity of the sum function associated with the Wiener process is established.
\section{\textbf{Random Fourier--Jacobi series associated with symmetric stable process}}
This section deals with the random series
\begin{equation}\label{2.1}
 \sum_{n=0}^\infty a_nA_n p_n^{(\gamma,\delta)}(y)
\end{equation}
in orthonormal Jacobi polynomials $p_n^{(\gamma,\delta)}(y),$ $\gamma,\delta>-1$ with random coefficients $A_n(\omega)$
 associated with the symmetric stable process $X(t,\omega)$ of index $\alpha \in[1,2]$ defined as in (\ref{2.2}).
The following lemma establishes that the random coefficients $A_n(\omega)$ are not independent.
\begin{theorem}\label{L1.1}
If $X(t,\omega),t \in \mathbb{R}$ is a symmetric stable process of index $\alpha \in [1,2],$ then the random variables $A_n(\omega)$ associated with $X(t,\omega)$ are not independent, for $\gamma,\delta>-1$ and $\eta,\tau \geq 0.$
\end{theorem}
\begin{proof}
 To prove $A_n(\omega)$ are not independent, it is sufficient to show that the characteristic function of $\Big(A_n(\omega)+A_m(\omega)\Big)$ is not same as the product of the characteristic functions of $A_n(\omega)$ and $A_m(\omega).$ It is known that the characteristic function of the stochastic integral $\int\limits_{-1}^1 f(t)dX(t,\omega)$ is
\begin{equation}
 \exp\Big(-C|x|^\alpha\int_{-1}^1|f(t)|^\alpha dt\Big),
\end{equation}
  where $C$ is a constant.
So the characteristic function of $A_n(\omega)$ is
\begin{equation*}
\exp\Big(-C|x|^\alpha\int_{-1}^1\Big|p_n^{(\gamma,\delta)}(t)\rho^{(\eta,\tau)}(t)\Big|^\alpha dt\Big).
\end{equation*}
Now, the characteristic function of sum of the random variables $(A_n(\omega)+A_m(\omega))$ i.e. the characteristic function of
\begin{equation*}
\int_{-1}^1 \Big\{p_n^{(\gamma,\delta)}(t)+p_m^{(\gamma,\delta)}(t)\Big\}\rho^{(\eta,\tau)}(t)dX(t,\omega)
\end{equation*}
is
\begin{equation*}
\exp\Big(-C|x|^\alpha\int_{-1}^1\Big|\Big(p_n^{(\gamma,\delta)}(t)+p_m^{(\gamma,\delta)}(t)\Big)\rho^{(\eta,\tau)}(t)\Big|^\alpha dt\Big).
\end{equation*}
The product of characteristic functions of random variables $A_n(\omega)$ and $A_m(\omega)$ is
\begin{eqnarray*}
&&\exp\Big(-C|x|^\alpha\int_{-1}^1\Big|p_n^{(\gamma,\delta)}(t)\rho^{(\eta,\tau)}(t)\Big|^\alpha dt\Big).\exp\Big(-C|x|^\alpha\int_{-1}^1\Big|p_m^{(\gamma,\delta)}(t)\rho^{(\eta,\tau)}(t)\Big|^\alpha dt\Big)\\
&=&\exp\Big(-C|x|^\alpha\int_{-1}^1\Big(\Big|p_n^{(\gamma,\delta)}(t)\rho^{(\eta,\tau)}(t)\Big|^\alpha +\Big|p_m^{(\gamma,\delta)}(t)\rho^{(\eta,\tau)}(t)\Big|^\alpha\Big) dt\Big),
\end{eqnarray*}
which is clearly not equal to the characteristic function of the sum $\Big(A_n(\omega)+A_m(\omega)\Big).$ Hence, $A_n(\omega)$ are not independent random variables, for $\gamma,\delta>-1,$ and $\eta,\tau \geq 0.$
\end{proof}

Theorem \ref{T2.1} below is on the convergence of
 the random Fourier--Jacobi series (\ref{2.1}). The convergence of the series (\ref{2.1}) is established in the sense of probability.
In fact,
a sequence of random variables $X_n$ is said to converge in probability to a random variable $X,$ if
$\lim\limits_{n\rightarrow \infty} P(|X_n - X|>\epsilon) = 0, \; \epsilon>0.$
\begin{theorem}\label{T2.1}
Let $X(t, \omega),t \in \mathbb{R}$ be a symmetric stable process of index $\alpha \in(1,2]$ and $A_n(\omega)$ be defined as in (\ref{2.2}). If $\gamma,\delta>-1$ and $\eta,\tau \geq 0$ satisfy
\begin{equation}\label{2.4}
\Big|\eta- \frac{\gamma}{2} -\frac{1}{2} + \frac{1}{p}\Big|< min \Big(\frac{1}{4}, \frac{1}{2}+\frac{1}{2}\gamma\Big),
\end{equation}
\begin{equation}\label{2.5}
\Big|\tau-\frac{\delta}{2}-\frac{1}{2} + \frac{1}{p} \Big|< min \Big(\frac{1}{4},\frac{1}{2}+\frac{1}{2}\delta\Big),
\end{equation}
and $a_n$ are the Fourier--Jacobi coefficients of $f\in L_{[-1,1]}^{p,(\eta,\tau)}$ defined as in (\ref{1.3}), then the random Fourier--Jacobi series (\ref{2.1}) converges in probability to the stochastic integral
\begin{equation}\label{2.3}
\int_{-1}^1 f(y,t)\rho^{(\eta,\tau)}(t)dX(t,\omega),
\end{equation}
 for $p \geq \alpha>1.$
\end{theorem}
The proof of this theorem requires the following result:
\begin{lemma} \label{l2.1} \cite{NPM}
Let $f(t)$ be any function in $L_{[a,b]}^p,\; p\geq 1$
and $X(t,\omega)$ be a symmetric stable process of index  $\alpha,$ for $1 \leq \alpha \leq 2.$ Then for all $\epsilon >0,$
\begin{equation*}
P\Bigg(\Bigg|\int_a^b f(t)dX(t,\omega)\Bigg|> \epsilon\Bigg)
\leq
\frac{C2^{\alpha+1}}{(\alpha+1)\epsilon'^{\alpha}}
\int_a^b|f(t)|^\alpha dt,
\end{equation*}
where $\epsilon' < \epsilon$ and $C$ is a positive constant, if $p \geq \alpha.$
\end{lemma}
\textbf{Proof of Theorem \ref{T2.1}}.\\
Let
\begin{equation}\label{2.6}
\mathbf{S}_n^{(\gamma,\delta)}(f,y,\omega):=
\sum_{k=0}^n a_k
A_k(\omega)p_k^{(\gamma, \delta)}(y)
\end{equation}
 be the $n$th partial sum of the random Fourier--Jacobi series (\ref{2.1}).
The integral form of (\ref{2.6}) is
\begin{eqnarray*}
\mathbf{S}_n^{(\gamma,\delta)}(f,y,\omega)&:=& \sum_{k=0}^n a_k \Bigg(\int_{-1}^1 p_k^{(\gamma,\delta)}(t)\rho^{(\eta,\tau)}(t)dX(t,\omega) \Bigg) p_k^{(\gamma,\delta)}(y)\\
&=& \int_{-1}^1 \sum_{k=0}^n a_k p_k^{(\gamma,\delta)}(t) p_k^{(\gamma,\delta)}(y)\rho^{(\eta,\tau)}(t)dX(t,\omega)\\
&=&\int_{-1}^1 \mathbf{s}_n^{(\gamma,\delta)}(f,y,t)\rho^{(\eta,\tau)}(t)dX(t,\omega),
\end{eqnarray*}
where
\begin{equation*}
\mathbf{s}_n^{(\gamma,\delta)}(f,y,t):=\sum\limits_{k=0}^n a_k p_k^{(\gamma,\delta)}(y)p_k^{(\gamma,\delta)}(t).
\end{equation*}
\\With the help of Lemma \ref{l2.1},
\begin{eqnarray*}
&&P\Bigg(\Bigg|\int_{-1}^1f(y,t)\rho^{(\eta,\tau)}(t)dX(t,\omega)-\mathbf{S}_n^{(\gamma,\delta)}(f,y,\omega)\Bigg| > \epsilon\Bigg)\\
&=& P\Bigg(\Bigg|\int_{-1}^1f(y,t)\rho^{(\eta,\tau)}(t)dX(t,\omega)-\int_{-1}^1 \mathbf{s}_n^{(\gamma,\delta)}(f,y,t)\rho^{(\eta,\tau)}(t)dX(t,\omega)\Bigg|>\epsilon \Bigg)\\
&\leq& \frac{C2^{\alpha+1}}{(\alpha+1)\epsilon'^{\alpha}}
 \int_{-1}^1\Big|\Big(f(y,t)-\mathbf{s}_n^{(\gamma,\delta)}(f,y,t)\Big)\rho^{(\eta,\tau)}(t)\Big|^\alpha dt,\; for\; \epsilon'< \epsilon.
\end{eqnarray*}
For $f \in L_{[-1,1]}^{p,(\eta,\tau)},\; p>1,$ if the weights $\gamma,\delta,\eta,\tau$ satisfy the conditions (\ref{2.4}), (\ref{2.5}), then
the $n$th partial sum $\mathbf{s}_n^{(\gamma,\delta)}(f,t)$ converges to $f(t)$
(by Theorem 1 in \cite{15A}).
Hence, for $p \geq \alpha >1,$
\begin{equation*}
\lim_{n \rightarrow \infty}\int_{-1}^1\Big|\Big(f(y,t)-\mathbf{s}_n^{(\gamma,\delta)}(f,y,t)\Big)\rho^{(\eta,\tau)}(t)\Big|^\alpha dt=0.
\end{equation*}
This implies, the convergence of the random Fourier--Jacobi series (\ref{2.1}) to the stochastic integral (\ref{2.3}) in probability,
for $p \geq \alpha>1.$

The convergence of random Fourier--Jacobi series (\ref{2.1}) associated with the symmetric stable process of index $\alpha=1$ is demonstrated by the following theorem.
\begin{theorem}\label{T2.2}
Let $X(t,\omega)$ be a symmetric stable process of index $\alpha=1.$ Let $a_n$ and $A_n(\omega)$ as defined in (\ref{1.3}) and (\ref{2.2}) be the Fourier--Jacobi coefficients of some function $f$ in $L_{[-1,1]}^{1,(\eta,\tau)},\;\eta,\tau \geq 0,$ and $X(t,\omega),$ respectively. Then the random series (\ref{2.1}) converges in probability to the stochastic integral (\ref{2.3})
provided
\begin{equation}\label{2.7}
\gamma-\eta \geq 0,\; and \;\delta-\tau \geq 0.
\end{equation}
\end{theorem}
\begin{proof}
The proof follows the steps of Theorem \ref{T2.1} under the conditions (\ref{2.7}) for $\gamma,\delta,\eta,\tau$ and uses the result of
\cite{MGN}.
\end{proof}
\section{\textbf{Random Fourier--Jacobi series associated with Wiener process}}
In this section, we consider the stochastic process $X(t,\omega)$ to be the Wiener process $W(t,\omega),\;t \geq 0$ and
the $n$th degree polynomials
\begin{equation}\label{0.4}
q_n^{(\gamma,\delta)}(t):=p_n^{(\gamma,\delta)}(2t-1),\; n \in \mathbb{N}\cup {0},\gamma,\delta>-1
\end{equation}
as the orthogonal polynomials instead of the polynomials $\varphi_n(t)$ in the random series (\ref{0.2}). These $q_n^{(\gamma,\delta)}(t)$ are orthogonal in the interval $[0,1]$ and form a complete orthonormal set in the interval $[0,1]$ with respect to the weight
\begin{equation*}
\sigma^{(\gamma,\delta)}(t):={(1-t)}^{\gamma}t^\delta,\;\gamma,\delta >-1.
\end{equation*}
We known that the stochastic integral
\begin{equation}\label{0.3}
\int\limits_a^b f(t)dW(t,\omega)
\end{equation}
exists in quadratic mean, for $f \in L^2_{[a,b]}$ \cite{NPM}.
The random sequence $\{X_n\}_{n=0}^\infty$ is said to be convergent in quadratic mean to a random variable $X$, if
\begin{equation*}
\lim\limits_{n \rightarrow \infty}E\Big(|X_n-X|^2\Big)=0.
\end{equation*}
The stochastic integral (\ref{0.3}) is normally distributed random variable with
mean zero and finite variance $\int\limits_a^b |f(t)|^2 dt,$ if $f(t)$ is a function in $ L_{[a,b]}^2$ (c.f. \cite[Lukacs,~p.~148]{22}).
The $q_n^{(\gamma,\delta)}(t)\sigma^{(\eta,\tau)}(t)$ remains continuous, for $\eta,\tau \geq 0,$ and hence the stochastic integrals
\begin{equation}\label{2.8}
B_n(\omega):=\int\limits_0^1 q_n^{(\gamma,\delta)}(t)\sigma^{(\eta,\tau)}(t)dW(t,\omega)
\end{equation}
with weight function $\sigma^{(\eta,\tau)}(t),\eta,\tau \geq 0$ exist in quadratic mean. These $B_n(\omega)$ are random variables with mean zero and finite variance.
The following lemma proves the independence of random variables $B_n(\omega).$
\begin{theorem}\label{L2.2}
If $X(t,\omega)$ is the Wiener process $W(t,\omega),t \geq 0,$ then the random variables $B_n(\omega)$ associated with $W(t,\omega)$ are independent.
\end{theorem}
\begin{proof}
The Wiener process $W(t,\omega)$ has orthogonal increments and if $f,g\in L_{[a,b]}^2,$ then by Doob \cite[p.~427]{d1},
\begin{equation*}
E\Big(\int_a^b f(t)dW(t,\omega) \overline{\int_a^b g(t) dW(t,\omega)}\Big)=\int_a^b f(t)\overline{g(t)}dt,
\end{equation*}
 where $\overline{g(t)}$ is the complex conjugate of $g(t).$\\
Thus, for $t \in [0,1],$
\begin{eqnarray*}
E\Big(B_n(\omega)\overline{B_m(\omega)}\Big)&=&E\Bigg(\int_0^1 q_n^{(\gamma,\delta)}(t)\sigma^{(\eta,\tau)}(t)dW(t,\omega)\overline{\int_0^1 q_m^{(\gamma,\delta)}(t)\sigma^{(\eta,\tau)}(t)dW(t,\omega)}\Bigg)\\
&=&\int_0^1 q_n^{(\gamma,\delta)}(t)\sigma^{(\eta,\tau)}(t)\overline{q_m^{(\gamma,\delta)}(t)\sigma^{(\eta,\tau)}(t)}dt\\
&=&\int_0^1 q_n^{(\gamma,\delta)}(t)q_m^{(\gamma,\delta)}(t)\{\sigma^{(\eta,\tau)}(t)\}^2dt.
\end{eqnarray*}
Let
$|\sigma^{(\eta,\tau)}(t)| \leq C$ in $[0,1]$ as
$\sigma^{(\eta,\tau)}(t)$ is bounded in $[0,1].$
Hence
\begin{eqnarray*}
E\Big(B_n(\omega)\overline{B_m(\omega)}\Big)&=&\int_0^1 q_n^{(\gamma,\delta)}(t)q_m^{(\gamma,\delta)}(t)\{\rho^{(\eta,\tau)}(t)\}^2dt\\
&\leq & C \int_0^1 q_n^{(\gamma,\delta)}(t)q_m^{(\gamma,\delta)}(t)\rho^{(\eta,\tau)}(t)dt=0,
\end{eqnarray*}
where $C$ is a $+$ve constant.
This proves the fact that $\{B_n(\omega)\}_{n=0}^\infty$ is a sequence of independent random variables, for $\gamma,\eta,\tau \geq 0$ and $t \in[0,1].$
\end{proof}

Now consider the random series
\begin{equation}\label{2.9}
\sum_{n=0}^\infty b_n B_n(\omega)q_n^{(\gamma,\delta)}(y),
\end{equation}
where $B_n(\omega)$ are defined as in (\ref{2.8}),
$b_n$ are scalars defined by
\begin{equation}\label{2.10}
b_n:= \int_0^1 f(t)q_n^{(\gamma,\delta)}(t)\sigma^{(\gamma,\delta)}(t)dt.
\end{equation}
 The $B_n(\omega)$ and $b_n$ are called the modified Fourier--Jacobi coefficients of $W(t,\omega)$ and the function $f$ respectively.
If $b_n$ as defined in (\ref{2.10})
are the Fourier--Jacobi coefficients
of function $f \in L_{[0,1]}^{2,(\eta,\tau)},\eta,\tau \geq 0$ with respect to the Jacobi polynomials $q_n^{(\gamma,\delta)}(t),\; \gamma,\delta>-1,$ then it can be shown that the random series (\ref{2.9}) converges in quadratic mean to the stochastic integral
\begin{equation}\label{2.14}
\int\limits_0^1 f(y,t)\sigma^{(\eta,\tau)}(t)dW(t,\omega).
\end{equation}
The proof of this result needs the following theorem, which is a modified form of the Theorem 1 in \cite{15A}.
\begin{theorem}\label{t2.3}
If $f \in L_{[0,1]}^{2,(\eta,\tau)}$ and the following conditions are satisfied by $\eta,\tau \geq 0,\;\gamma,\;\delta> -1,$
\begin{eqnarray*}
\Big|\eta-\frac{1}{2}\gamma\Big| &<& min\Big(\frac{1}{4}, \frac{1}{2}+ \frac{1}{2}\gamma\Big)\\
\Big|\tau-\frac{1}{2}\delta\Big| &<& min\Big(\frac{1}{4}, \frac{1}{2}+ \frac{1}{2}\delta\Big),
\end{eqnarray*}
then
\begin{equation*}
\lim_{n\rightarrow \infty}\int_0^1\Big|\Big\{\mathbf{v}_n^{(\gamma,\delta)}(f,y)- f(y)\Big\}\sigma^{(\eta,\tau)}(y)\Big|^2 dy=0, \;for \; y \in [0,1],
\end{equation*}
where $\mathbf{v}_n^{(\gamma,\delta)}(f,y)$ is the $n$th partial sum of the Fourier--Jacobi series $\sum\limits_{n=0}^\infty b_n q_n^{(\gamma,\delta)}(y).$
\end{theorem}
The following theorem establishes the convergence of the random Fourier--Jacobi series (\ref{2.9}) in modified Jacobi polynomials $q_n^{(\gamma,\delta)}(t)$ associated with the Wiener process $W(t,\omega).$
 \begin{theorem}\label{T2.3}
Let $W(t,\omega),t \geq 0$ be the Wiener process and $B_n(\omega)$ be defined as in (\ref{2.8}). If the weights $\gamma,\delta >-1$ and $\eta,\tau \geq 0$ satisfy the conditions
\begin{equation}\label{2.13}
\Big|\eta-\frac{1}{2}\gamma\Big| < min\Big(\frac{1}{4}, \frac{1}{2}+ \frac{1}{2}\gamma\Big),\;\Big|\tau-\frac{1}{2}\delta\Big| < min\Big(\frac{1}{4}, \frac{1}{2}+ \frac{1}{2}\delta\Big),
\end{equation}
and $b_n$ are the Fourier--Jacobi coefficients of $f \in L_{[0,1]}^{2,(\eta,\tau)},$
then the random Fourier--Jacobi series (\ref{2.9}) converges to the integral (\ref{2.14}) in quadratic mean.
\end{theorem}
\begin{proof}
 Let
\begin{equation}\label{2.15.1}
T_n^{(\gamma,\delta)}(f,y,\omega):=
\sum_{k=0}^n b_k
B_k(\omega)q_k^{(\gamma, \delta)}(y),\;\gamma,\delta>-1
\end{equation} be the $n$th partial sum of the random Fourier--Jacobi series (\ref{2.9}).
The integral form of (\ref{2.15.1}) is
\begin{equation*}
\int_0^1 \mathbf{v}_n^{(\gamma,\delta)}(f,y,t)\sigma^{(\eta,\tau)}(t)dW(t,\omega),\; \eta,\tau \geq 0,
\end{equation*}
with
\begin{equation*}
\mathbf{v}_n^{(\gamma,\delta)}(f,y,t):=\sum\limits_{k=0}^n b_k q_k^{(\gamma,\delta)}(y)q_k^{(\gamma,\delta)}(t),
\end{equation*}
for $t \in [0,1].$
We know that (Lukacs \cite[p.~147]{22}) for $g \in L_{[a,b]}^2,$ if $W(t,\omega)$ is the Wiener process, then
\begin{equation}\label{2.21}
E\Big|\int_a^b g(t)dW(t,\omega)\Big|^2=\beta^2\int_a^b |g(t)|^2dt,
\end{equation}
where $\beta$ is a constant associated with the normal law of increment of the process $W(t,\omega),$ for $t \in [a,b].$
Hence, the equality
\begin{eqnarray*}
&&E\Big(\Big|\int_0^1 f(y,t)\sigma^{(\eta,\tau)}(t)dW(t,\omega)-T_n^{(\gamma,\delta)}(f,y,\omega)\Big|^2\Big)\\
&=& E\Big(\Big|\int_0^1 f(y,t)\sigma^{(\eta,\tau)}(t)dW(t,\omega)- \int_0^1 \mathbf{v}_n^{(\gamma,\delta)}(f,y,t)\sigma^{(\eta,\tau)}(t)dW(t,\omega)\Big|^2\Big)\\
&=&\beta^2\int_0^1 \Big|\Big(f(y,t)-\mathbf{v}_n^{(\gamma,\delta)}(f,y,t)\Big)\sigma^{(\eta,\tau)}(t)\Big|^2dt.
\end{eqnarray*}
 If $f \in L_{[0,1]}^{2,(\eta,\tau)}$ and
 $\gamma,\delta,\eta,\tau$ satisfy the conditions in (\ref{2.13}),
then by Theorem \ref{t2.3},
\begin{equation*}
\lim_{n \rightarrow \infty}\int\limits_0^1\Big|\Big(f(t)-\mathbf{v}_n^{(\gamma,\delta)}(f,t)\Big)\sigma^{(\eta,\tau)}(t)\Big|^2 dt =0.
\end{equation*}
Hence
\begin{equation*}
\lim\limits_{n \rightarrow \infty} \int\limits_0^1\Big|\Big(f(y,t)-\mathbf{v}_n^{(\gamma,\delta)}(f,y,t)\Big)\sigma^{(\eta,\tau)}(t)\Big|^2dt=0,
\end{equation*}
which implies convergence of the random series (\ref{2.9}) in quadratic mean to the integral (\ref{2.14}).
\end{proof}

The almost sure convergence of the random Fourier--Jacobi series (\ref{2.9}) is derived in Theorem \ref{T2.4}.
The Kolmogorov Theorem stated below is required to prove it.
\begin{theorem} \label{t2.4}(Kolmogorov Theorem) \\
Let ${(X_n)}_{n=1}^\infty$ be independent random variables with expected values $E[X_n]=\mu_n$ and variances $Var(X_n)=\sigma_n^2,$ such that $\sum\limits_{n=1}^\infty \mu_n$ converges in $\mathbb{R}$ and $\sum\limits_{n=1}^\infty \sigma_n^2$ converges in $\mathbb{R}.$ Then $\sum\limits_{n=1}^\infty X_n$ converges in $\mathbb{R}$ almost surely.
\end{theorem}
%Now the upgrade of the convergence in quadratic mean of the random Fourier-Jacobi series (\ref{2.9}) to almost sure convergence is discussed below.

The following facts about $p_n^{(\gamma,\delta)}(y)$ will be useful to establish the almost sure convergence of the series (\ref{2.9}).
We know that the Jacobi polynomials
 \begin{equation}\label{2.16}
|p_n^{(\gamma,\delta)}(y)| \leq C'n^{-\frac{1}{2}}{(1-y+n^{-2})}^{-\frac{1}{2}\gamma-\frac{1}{4}},\; 0\leq y \leq 1,\;\gamma, \delta >-1,
\end{equation}
where $C'$ is a constant independent of $y$ and $n$ \cite[p.~167]{1}.

The equality
\begin{center}
$p_n^{(\gamma,\delta)}(y)
={(-1)}^n p_n^{(\gamma,\delta)}(-y)$ (see \cite[p.~71]{1})
\end{center}
extends the inequality (\ref{2.16}), to hold for all $y \in [-1,1].$
This can be applied for the modified orthonormal Jacobi polynomials $q_n^{(\gamma,\delta)}(y)$ in the interval $[0,1],$ and we obtain
\begin{eqnarray*}
|q_n^{(\gamma,\delta)}(y)|&\leq &  C'n^{-\frac{1}{2}}{(1-y+n^{-2})}^{-\frac{1}{2}\gamma-\frac{1}{4}}\\
&=&\frac{C'}{{[(1-y)n^2+1]}^{\gamma/2+1/4}n^{-\gamma}}\\
&=&\frac{C'n^\gamma}{{[(1-y)n^2+1]}^{\gamma/2+1/4}}
\end{eqnarray*}
where $1/[(1-y)n^2+1]$ is bounded by $1/2$ in $[0,1].$\\
Hence
\begin{equation*}
|q_n^{(\gamma,\delta)}(y)| \leq \frac{C'n^\gamma}{{2}^{\gamma/2+1/4}}
\end{equation*}
\begin{equation}\label{2.18}
i.e \; |q_n^{(\gamma,\delta)}(y)| \leq C n^\gamma,
\end{equation}
where $C$ is a constant independent of $y$ and $n.$
\begin{theorem}\label{T2.4}
Let $W(t,\omega),\;t \geq 0$ be the Wiener process and $B_n(\omega)$ be as defined in (\ref{2.8}).
 If $\gamma,\delta>-1$ and $\eta,\tau \geq 0$ satisfy the conditions (\ref{2.13}) in Theorem \ref{T2.3} and
\begin{equation}\label{2.20}
\sum_{n=0}^\infty \{n^{(2\gamma)}|b_n|\}^2 < \infty,
\end{equation}
then the series
(\ref{2.9}) converges almost surely
to the stochastic integral (\ref{2.14}), for the scalars $b_n$ are the Fourier--Jacobi coefficients of the function $f \in L_{[0,1]}^{2,(\eta,\tau)},\eta,\tau \geq 0.$
\end{theorem}
\begin{proof}
 We know that
 the independent random variables $B_n(\omega)$ are normally distributed with mean zero and finite variance. Hence each $b_nB_n(\omega)q_n^{(\gamma,\delta)}(y),$ for $n=1,2,\dots$ are normally distributed, independent random variables with mean zero and finite variance.
Now using the identity (\ref{2.21}), the sum of the variance of these random variables is
\begin{eqnarray*}
\sum_{n=0}^\infty E\Big|b_nB_n(\omega)q_n^{(\gamma,\delta)}(y)\Big|^2
&=&\sum_{n=0}^\infty E\Big| b_n \int_0^1 q_n^{(\gamma,\delta)}(t)\sigma^{(\eta,\tau)}(t)dW(t,\omega)q_n^{(\gamma,\delta)}(y)\Big|^2\\
&=&\sum_{n=0}^\infty E\Big|\int_0^1 b_n q_n^{(\gamma,\delta)}(y) q_n^{(\gamma,\delta)}(t)\sigma^{(\eta,\tau)}(t)dW(t,\omega)\Big|^2\\
&=&\sum_{n=0}^\infty  \int_0^1\Big|b_n q_n^{(\gamma,\delta)}(y) q_n^{(\gamma,\delta)}(t)\sigma^{(\eta,\tau)}(t)\Big|^2 dt.
%&\leq&\sum_{n=0}^\infty \int_0^1 |a_n|^2 | q_n^{(\gamma,\delta)}(y) q_n^{(\gamma,\delta)}(t)\sigma^{(\eta,\tau)}(t)|^2 dt.
\end{eqnarray*}
Since $q_n^{(\gamma,\delta)}(t)$ are bounded by $Cn^{\gamma}$ (inequality (\ref{2.18})) and $\sigma^{(\eta,\tau)}(t)$ is bounded, for $\eta,\tau\geq 0,$
%and $\sigma^{(\gamma,\delta)}(t)$ bounded by $1.$
we have the inequality
\begin{eqnarray*}
\sum_{n=0}^\infty E\Big | b_n B_n(\omega)q_n^{(\gamma,\delta)}(y)\Big|^2
&\leq& K \sum_{n=0}^\infty {(|b_n |n^{2\gamma})}^2,
\end{eqnarray*}
which will be finite, for $f \in L_{[0,1]}^{2,(\eta,\tau)},$ if $\sum\limits_{n=0}^\infty {(|b_n|n^{2\gamma})}^2$ is finite.
Now by Kolmogorov's Theorem (Theorem \ref{t2.4}), the series (\ref{2.9}) converges
to the integral $\int\limits_0^1 f(y,t)\sigma^{(\eta,\tau)}(t)dW(t,\omega)$
almost surely in $\mathbb{R},$ if $\gamma,\delta>-1,\eta,\tau \geq 0$ satisfy the conditions (\ref{2.13}) in Theorem \ref{T2.3}.
\end{proof}
\section{\textbf{Continuity property of the sum functions}}
\subsection{\textbf{Sum function associated with symmetric stable process}}
The sum function of the random Fourier--Jacobi series (\ref{2.1}) associated with the symmetric stable process $X(t,\omega)$ is shown to be
weakly continuous in probability.
We know that a function $f(t,\omega)$ is said to be weakly continuous in probability at $t = t_0,$
if for all $\epsilon > 0,$
$\lim\limits_{h \rightarrow 0}P(\vert f(t_0 + h, \omega) - f(t_0, \omega) \vert > \epsilon) = 0.$
If a function $f(t,\omega)$ is weakly continuous at every $t_0 \in[a,b],$ then the function $f(t,\omega)$ is said to be weakly continuous in
probability in the closed interval
$[a, b].$
The proof of this result required the following lemma.
%The following lemma is required to proof of this result. 
\begin{lemma} \label{l2.2} \cite[p.~37]{27}
If $f$ is periodic or in $L_{[a,b]}^p,1 \leq p<\infty$ or continuous function, then the integral
\begin{equation*}
\Bigg\{\int_a^b \Big|f(x+t)-f(x)\Big|^pdx\Bigg\}^{1/p}
\end{equation*}
 tends to $0$ as $t$ tends to $0.$
\end{lemma}
\begin{theorem} \label{T2.3.1}
In the Theorem \ref{T2.1} and Theorem \ref{T2.2}, the sum function (\ref{2.3}) of the random Fourier--Jacobi series (\ref{2.1}) associated with the symmetric stable process
according to the respective conditions of $\gamma,\delta,\eta, \tau$
 are weakly continuous in probability.
\end{theorem}
%The result in Lemma \ref{l2.2} can be extended to the functions $f$ in $L_{[-1,1]}^{p,\eta,\tau}$ by Fatous's lemma.\\
%\textbf{Proof of Theorem \ref{T2.3.1}}:\\
\begin{proof}
With the help of Lemma \ref{l2.1},
\begin{eqnarray*}
&&P\Bigg(\Big|\int_{-1}^1f(x,t)\rho^{(\eta,\tau)}(t)dX(t,\omega))-
\int_{-1}^1f(y,t)\rho^{(\eta,\tau)}(t)dX(t,\omega)\Big|>\epsilon \Bigg)\\
&\leq&
\frac{C2^{\alpha+1}}{(\alpha+1)\epsilon^{'\alpha}}
\int_{-1}^1\Big|\Big(f(x,t)-f(y,t)\Big) \rho^{(\eta,\tau)}(t)\Big|^\alpha dt,
\end{eqnarray*}
where $0<\epsilon'<\epsilon$ and $p \geq \alpha \in [1,2].$
Since, the weight $\rho^{(\eta,\tau)}(t)$ is bounded,
the integral
\begin{equation*}
\int_{-1}^1\Big|\Big(f(x,t)-f(y,t)\Big)\Big|^\alpha dt
\end{equation*}
tends to $0$ as $y \rightarrow x,$ by Lemma \ref{l2.2}.
This confirms that the sum function (\ref{2.3}) is weakly continuous in probability.
\end{proof}
\subsection{\textbf{Sum function associated with Wiener process}}
\begin{theorem}
The sum function (\ref{2.14}) of the random Fourier--Jacobi series (\ref{2.9}) under the condition on $\gamma,\delta,\eta, \tau$ in Theorem \ref{T2.3} is continuous in quadratic mean.
\end{theorem}
\begin{proof}
By the use of (\ref{2.21}),
\begin{eqnarray*}
&&E\Big(\Big|\int_0^1 f(y,t)\sigma^{(\eta,\tau)}(t)dW(t,\omega)- f(x,t)\sigma^{(\eta,\tau)}(t)dW(t,\omega)\Big|^2\Big)\\
&=&E\Big(\Big|\int_0^1 \Big( f(y,t)-f(x,t)\Big)\sigma^{(\eta,\tau)}(t)dW(t,\omega)\Big|^2\Big)\\
&=& \beta^2 \int_0^1  \Big|\Big(f(y,t)-f(x,t)\Big)\sigma^{(\eta,\tau)}(t)\Big|^2dt.
\end{eqnarray*}
 For $\eta,\tau \geq 0,$
the Jacobi weight $\sigma^{(\eta,\tau)}(t)$ is bounded i.e $|\sigma^{(\eta,\tau)}(t)| \leq C,$ for some $c>0.$
Then
\begin{eqnarray*}
&&E\Big(\Big|\int_0^1 f(y,t)\sigma^{(\eta,\tau)}(t)dW(t,\omega)- f(x,t)\sigma^{(\eta,\tau)}(t)dW(t,\omega)\Big|^2\Big)\\
&& \leq \beta^2 C^2 \int_0^1 \Big|\Big(f(y,t)-f(x,t)\Big)\Big|^2dt.
\end{eqnarray*}
Hence by Lemma \ref{l2.2}, the right hand side tends to zero as $y \rightarrow x.$ This proves that the sum function (\ref{2.14}) in Theorem \ref{T2.3} is continuous in quadratic mean.
\end{proof}
The following theorem establishes the improvement of the continuity property of the sum function (\ref{2.14}) from quadratic mean to almost surely.
\begin{theorem}
The sum function (\ref{2.14}) of the random Fourier--Jacobi series (\ref{2.9}) is almost surely continuous, if
\begin{equation}\label{2.22}
\sum_{n=0}^\infty (n^\gamma|b_n|)< \infty,
\end{equation}
in addition to the conditions on $\gamma,\delta,\eta,\tau$ stated in Theorem \ref{T2.3}.
\end{theorem}
\begin{proof}
By Weistrass M-test, the series (\ref{2.9}) converges uniformly to a continuous function almost surely, for almost all $y \in [0,1],$
if
\begin{equation*}
\sum_{n=0}^\infty \Big|b_n B_n(\omega)q_n^{(\gamma,\delta)}(y)\Big|<\infty.
\end{equation*}
It is sufficient to show that
\begin{equation*}
\sum_{n=0}^\infty E\Big|b_n B_n(\omega)q_n^{(\gamma,\delta)}(y)\Big|<\infty.
\end{equation*}
Now\\
\begin{eqnarray*}
\sum_{n=0}^\infty E\Big|b_nB_n(\omega)q_n^{(\gamma,\delta)}(y)\Big|
%\leq \sum_{n=0}^\infty E|a_n| |A_n(\omega)| |q_n^{(\gamma,\delta)}(y)|.
&\leq & K\sum_{n=0}^\infty (|b_n| n^\gamma),\;where \; K \; is \; a\; constant,
\end{eqnarray*}
as $q_n^{(\gamma,\delta)}(y)$ are bounded by $Cn^\gamma$ and  $B_n(\omega)$ are bounded.
If the sum in the right hand side series is finite, then the almost sure continuity of the sum function (\ref{2.14}) is established, for almost all $y \in [0,1].$
\end{proof}
\section*{Remark}
In all our results the weights associated with the Jacobi polynomials are considered to be $\gamma,\delta>-1$ and $\eta,\tau \geq 0.$
The results that we obtained can be summarized as follows:\\
\begin{enumerate}
\item If $X(t,\omega),t \in \mathbb{R}$ is a symmetric stable process of the index $\alpha=1$ and
weights $\gamma,\delta,\eta,\tau$ satisfy
\begin{equation*}
\gamma-\eta \geq 0,\; and \;\delta-\tau \geq 0,
\end{equation*}
then the random Fourier--Jacobi series (\ref{2.1}) converges in probability to the integral (\ref{2.3}).
 \item  If the index $\alpha \in (1,2],$ and
\begin{eqnarray*}
\Big(\Big|\eta- \frac{\gamma}{2} -\frac{1}{2} + \frac{1}{p}\Big|&<& min \Big(\frac{1}{4}, \frac{1}{2}+\frac{1}{2}\gamma\Big),\\
\Big(\Big|\tau-\frac{\delta}{2}-\frac{1}{2} + \frac{1}{p} \Big|&<& min \Big(\frac{1}{4},\frac{1}{2}+\frac{1}{2}\delta\Big),
\end{eqnarray*}
are satisfied by the weights $\gamma,\delta,\eta,\tau,$
then the random Fourier--Jacobi series (\ref{2.1}) converges in probability to the integral (\ref{2.3}).
\item The sum function (\ref{2.3}) is weakly continuous in probability, if $X(t,\omega)$ is the symmetric stable process of index $ \alpha \in [1,2].$
\item  If $X(t,\omega)$ is the Wiener process $W(t,\omega),t\geq 0$ and
\begin{eqnarray*}
\Big(\Big|\eta- \frac{\gamma}{2}\Big|&<& min \Big(\frac{1}{4}, \frac{1}{2}+\frac{1}{2}\gamma \Big),\\
\Big(\Big|\tau-\frac{\delta}{2} \Big|&<& min \Big(\frac{1}{4},\frac{1}{2}+\frac{1}{2}\delta \Big),
\end{eqnarray*}
is satisfied
by $\gamma,\delta,\eta,\tau,$ then
 the random Fourier--Jacobi series (\ref{2.9}) converges in quadratic mean to the integral (\ref{2.14}).
\item The sum function (\ref{2.14}) associated with the Wiener process is continuous in quadratic mean.
\item In addition to the conditions as in (d) on $\gamma,\delta,\eta,\tau,$ if the Fourier--Jacobi coefficients $b_n$ of the function $f \in L_{[0,1]}^{2,(\gamma,\delta)}$ satisfy the strong condition
\begin{equation*}
\sum_{n=0}^\infty {(|b_n|n^{2\gamma})}^2 < \infty,
\end{equation*}
then the random Fourier--Jacobi series (\ref{2.9}) is convergent almost surely to the stochastic integral (\ref{2.14}).
\item The sum function (\ref{2.14}) associated with the Wiener process is almost surely continuous,
if
\begin{equation*}
\sum_{n=0}^\infty |b_n|n^\gamma < \infty.
\end{equation*}
\end{enumerate}

\section{Acknowledgements}
This research work was supported by University Grant Commission (National Fellowship with letter no-F./2015-16/NFO-2015-17-OBC-ORI-33062).

\end{document}